\documentclass[12pt]{article}

\usepackage[margin=1.0in]{geometry}
\usepackage{setspace}
\usepackage{amsmath}
\usepackage{amssymb}
\usepackage{amsthm}
\usepackage{color}

\DeclareMathOperator{\conv}{conv}
\DeclareMathOperator{\diag}{diag}
\DeclareMathOperator{\Diag}{Diag}

\newtheorem{theorem}{Theorem}

\newtheorem{lemma}{Lemma}

\newtheorem{conjecture}{Conjecture}

\renewcommand{\Dot}{\bullet}
\newcommand{\Rbb}{\mathbb{R}}
\newcommand{\Sbb}{{\mathbb{S}}}
\newcommand{\tran}{^T}
\newcommand{\leaveout}[1]{}
\newcommand{\set}[1]{\{ #1 \}}
\newcommand{\suchthat}{\,:\,}

\newcommand{\gesem}{\succeq}
\newcommand{\gtsem}{\succ}

\newcommand{\alphahat}{\hat{\alpha}}
\newcommand{\alphabar}{{\bar\alpha}}
\newcommand{\alphaprime}{\alpha^\prime}
\newcommand{\lam}{\lambda}

\newcommand{\Xbar}{\bar{X}}
\newcommand{\xbar}{\bar{x}}
\newcommand{\SDP}{{\rm PSD}}
\newcommand{\RLT}{{\rm RLT}}
\newcommand{\PER}{{\rm PER}}
\newcommand{\CP}{{\rm CP}}
\newcommand{\DNN}{{\rm DNN}}

\newcommand{\Hcal}{{\cal H}}
\newcommand{\Ccal}{{\cal C}}
\newcommand{\Dcal}{{\cal D}}

\title{Quadratic Optimization with Switching Variables: \\
The Convex Hull for $n=2$}

\author{%
Kurt M. Anstreicher%
\thanks{Department of Business Analytics, University of Iowa,
Iowa City, IA, 52242-1994, USA. Email: {\tt kurt-anstreicher@uiowa.edu}.}%
\and
Samuel Burer%
\thanks{Department of Business Analytics, University of Iowa,
Iowa City, IA, 52242-1994, USA. Email: {\tt samuel-burer@uiowa.edu}.}%
}

\date{February 9, 2020}

\begin{document}

\maketitle

\begin{abstract}
We consider quadratic optimization in variables $(x,y)$ where $0\le x\le y$, and
$y\in\set{0,1}^n$.  Such binary $y$ are commonly refered to as {\em indicator\/} or
{\em switching\/} variables and occur commonly in applications.  One approach to such problems
is based on representing or approximating the convex hull of the set $\set{ (x,xx\tran, yy\tran)\suchthat 0\le
x\le y\in\set{0,1}^n}$.  A representation for the case $n=1$ is known and has been widely
used. We give an exact representation for the case $n=2$ by starting with a disjunctive
representation for the convex hull and then eliminating auxilliary variables
and constraints that do not change the projection onto the original variables.
An alternative derivation for this representation leads to an appealing conjecture for
a simplified representation of the convex hull for $n=2$ when the product term $y_1y_2$ is
ignored.

\end{abstract}

\section{Introduction}

This paper concerns quadratic optimization in variables $x\in\Rbb^n$
and $y\in\set{0,1}^n$, where $0\le x \le y$. The $y$ variables are
refered to as {\em indicator\/} or {\em switching\/} variables and
occur frequently in applications, including electrical power production
\cite{Frangioni.Gentile.2006}, constrained portfolio optimization
\cite{Frangioni.Gentile.2006,Gunluk.Linderoth.2010}, nonlinear machine
scheduling problems \cite{Akturk.et.al.2009} and chemical pooling
problems \cite{DAmbrosio.et.al.2011}. A typical feature of such problems
is that the objective function is separable in $x$ and $y$. In addition,
many applications do not involve the cross-terms $y_i y_j$ for $i \ne
j$.

One approach for such problems is to consider symmetric matrix variables
$X$ and $Y$ that replace the rank-1 matrices $xx\tran$ and $yy\tran$,
respectively. Using such variables, an objective of the form $c\tran x  +
x\tran Q x + y\tran Dy$ can be replaced by the linear function $c\tran
x + Q\Dot X + D\Dot Y$, where $(x,X,Y)$ should then be in
the set
\[
\Hcal := \conv\set{ (x,xx\tran, yy\tran)\suchthat 0\le
x\le y\in\set{0,1}^n}.
\]
 The problem is then to represent $\Hcal$ in
a computable manner. Note that, because $y$ is binary, $\diag(Y)$ captures $y$,
and in particular, when the cross-terms
$y_i y_j$ are not of interest, we may consider the simpler convex hull
\[
\Hcal' := \conv\set{ (x,xx\tran, y)\suchthat 0\le x\le y\in\set{0,1}^n}.
\]

For general $n$, determining computable representations of
$\Hcal$ and $\Hcal'$ is difficult. For example, even when $y$ is
fixed to $e$, the resulting convex hull, called $\text{QPB}$ in
\cite{Burer.Letchford.2009} for ``quadratic programming over the box,''
is intractable. When $n = 2$, an exact representation for QPB was given
in \cite{Anstreicher.Burer.2010}, but such a representation is not
known for $n \ge 3$. For general $n$, the paper
\cite{Dong.Linderoth.2013} studies valid inequalities for $\Hcal'$. For
the case $n=1$, $\Hcal = \Hcal'$ since there are no cross-terms, and
a computable representation was given in \cite{Gunluk.Linderoth.2010}
based on prior work in \cite{Frangioni.Gentile.2006}. This
representation has subsequently been used in a variety of applications;
see for example \cite{Hijazi.et.al.2017,Jabr.2012}. Several authors
have also studied the case when $n = 2$ but have focused on convexifying
in the space of $(x,y,t)$, where $t$ is a scalar associated with the epigraph
of a specially structured quadratic function, e.g., a convex quadratic one;
see \cite{Atamturk.Gomez.2018} and references therein.

In Section \ref{sec:n=1}, we consider the case of $n=1$ and 
reprove the representation of $\Hcal
= \Hcal'$ in a new way which incorporates standard ideas from the
literature on constructing strong semidefinite programming (SDP)
relaxations of quadratic programs. In particular, our proof can be
viewed as establishing that $\Hcal$ for $n=1$ is captured exactly by the
relaxation which uses the standard positive semidefinite (PSD) condition
along with the standard Reformulation--Linearization Technique (RLT)
constraints \cite{Sherali.Adams.1997}.

Our main result in this paper is a representation of $\Hcal$ for $n
= 2$, which we derive in several steps. Note that in this case
there is only a single cross-term $y_1 y_2$, and we can write
$\Hcal$ in the form
\[
\Hcal = \conv\set{ (x,xx\tran, y,y_1y_2)\suchthat 0\le
x\le y\in\set{0,1}^2}.
\]
First, in Section \ref{sec:disj}, we give a disjunctive
representation of $\Hcal$ that involves additional variables
$\alpha\in\Rbb^2$, $\beta\in\Rbb^2$. Then in Section \ref{sec:beta}
we project out $\beta$ by replacing a single PSD constraint with
four PSD constraints. The primary effort in the paper
occurs in Section \ref{sec:psd}, where we show that it is in fact
only necessary to impose one of these four PSD constraints in order
to represent $\Hcal$. This analysis is relatively complex due to
the fact that we are attempting to characterize the projection of
$(x,X,y,Y_{12},\alpha)$ onto $(x,X,y,Y_{12})$ where the constraints on
$(x,X,y,Y_{12},\alpha)$ include PSD conditions. If all constraints
on $(x,X,y,Y_{12},\alpha)$ were linear, we
could use standard polyhedral techniques to perform this projection.
However, since our case includes PSD conditions,
we are unaware of any general methodolgy for
characterizing such a projection, and therefore our proof technique is
tailored to the structure of $\Hcal$ for $n = 2$.

Finally, in Section \ref{sec:conjecture}, we describe an
alternative derivation for the representation of $\Hcal$ obtained in
Section \ref{sec:psd}. This derivation provides another interpretation
for the single remaining PSD condition and also leads to a conjecture
that a weaker PSD condition is sufficient to characterize $\Hcal'$ for
$n = 2$. If true, this conjecture would establish that $\Hcal'$ can be
represented using PSD, RLT, and simple linear conditions derived from
the binary nature of $y$, thus generalizing the results of Section
\ref{sec:n=1} for $n = 1$ as well as the representation of QPB
for $n=2$ from \cite{Anstreicher.Burer.2010}. This conjecture is supported
by extensive numerical computations but remains unproved.

\medskip

\noindent{\bf Notation.} We use $e$ to denote a vector of arbitrary
dimension with each component equal to one, and $e_i$ to denote an
elementary vector with all components equal to zero exept for a one in
component $i$. For symmetric matrices $X$ and $Y$, $X\gesem Y$ denotes
that $X-Y$ is positive semidefinite (PSD) and $X\gtsem Y$ denotes that
$X-Y$ is positive definite. The vector whose components are those
of the diagonal entries of a matrix $X$ is denoted $\diag(X)$.
The convex hull of a set is denoted $\conv\set{\cdot}$.

\section{\boldmath The convex hull for $n =1$}\label{sec:n=1}

In this section we consider the representation of $\Hcal$ for $n = 1$;
note that $\Hcal = \Hcal'$ in this case. The representation given
in Theorem \ref{thm:n=1} below is known, but to our knowledge the proof
given here is new. We define
\[
    \PER := \left\{
        (\alpha,\beta,\gamma) \in \Rbb \times \Rbb \times \Rbb :
        \begin{array}{l}
        \alpha^2 \le \beta \gamma \\
        0 \le \beta \le \alpha \le \gamma
        \end{array}
    \right\}
\]
to be the so-called {\em perspective cone\/} in $\Rbb^3$.
In particular, the constraint $\alpha^2 \le \beta \gamma$ is called a {\em perspective} constraint in the literature \cite{Gunluk.Linderoth.2010}.

\begin{theorem} \label{thm:n=1}
  For $n = 1$, $\Hcal = \Hcal' = \{ (x_1, X_{11}, y_1) \in \PER : y_1 \le 1 \}$.
\end{theorem}

\begin{proof}
Let $t_1=1-y_1$. Then the constraints  $0\le x_1\le y_1, y_1\in\set{0,1}$ can be written in the form $x_1+s_1+t_1=1$, $x_1\ge 0$, $s_1\ge 0$, $t_1\in\set{0,1}$.  By relaxing
the rank-one matrix $(1,x_1,s_1,t_1)^T (1,x_1,s_1,t_1)$ we obtain a matrix
\begin{equation}\label{eq:W:n=1}
W=\begin{pmatrix}
 1 & x_1 & s_1 & t_1 \\
 x_1 & X_{11} & Z_{11} & 0\\
 s_1 & Z_{11} & S_{11} & 0 \\
 t_1 & 0 & 0 & t_1
 \end{pmatrix},
\end{equation}
where we are using the fact that, for binary $t_1$, it holds that $t_1^2=t_1$ and $x_1t_1=s_1t_1=0$.  Multiplying $x_1+s_1+t_1=1$ in turn
 by the variables $x_1$ and $s_1$, we next obtain the RLT constraints
 $X_{11}+Z_{11}=x_1$ and $S_{11}+Z_{11}=s_1$. Let
 \begin{eqnarray*}
 \Ccal&=&\conv\set{(1, x_1,s_1,t_1)\tran (1, x_1,s_1,t_1)\suchthat x_1+s_1+t_1=1, x_1\ge 0, s_1\ge 0, t_1\in\set{0,1}},\\
\Dcal&=&\set{W\in\DNN\suchthat x_1+s_1+t_1=1, X_{11}+Z_{11}=x_1, S_{11}+Z_{11}=s_1},
\end{eqnarray*}
where the matrix $W$ in the definition of $\Dcal$ has the form \eqref{eq:W:n=1}, and $\DNN$ denotes the cone of
doubly nonnegative matrices, that is, matrices that are both componentwise nonnegative and PSD.
We claim that $\Ccal=\Dcal$.  The inclusion $\Ccal\subset\Dcal$ is obvious by standard SDP-relaxation techniques. However, from
\cite[Corollary 2.5]{Burer.2009} we know that
\[
\Ccal =
\{W\in\CP\suchthat x_1+s_1+t_1=1,\ X_{11}+S_{11}+t_1+2Z_{11}=1\},
\]
where $\CP$ denotes the cone of completely positive matrices, that is, matrices that can be represented as a sum of
nonnegative rank-one matrices. Note that  $X_{11}+S_{11}+t_1+2Z_{11}=1$ is the  ``squared'' constraint obtained by
 substituting appropriate variables into the expression  $(x_1+s_1+t_1)^2=1$. Then $\Ccal=\Dcal$ follows from the facts that since $W$ is $4\times 4$,
 $W\in\CP \iff W \in\DNN$, and the constraints  $x_1+s_1+t_1=1$, $X_{11}+Z_{11}=x_1$ and $S_{11}+Z_{11}=s_1$ together imply $X_{11}+S_{11}+t_1+2Z_{11}=1$.

 From $\Ccal=\Dcal$ we conclude that $\conv\set{(x_1,x_1^2,y_1)\suchthat 0\le x_1\le y_1, y_1\in\set{0,1}}=
  \set{(x_1, X_{11}, 1-t_1) \suchthat x_1+s_1+t_1=1, X_{11}+Z_{11}=x_1, S_{11}+Z_{11}=s_1, W\in\DNN}$.
 To complete the proof we will simplify the condition that $W\gesem 0$. Note that
\[
 W
=\begin{pmatrix} 1 & 0 &0\\ 0 & 1 &0\\ 1 & -1 & -1\\ 0 & 0 & 1 \end{pmatrix}
 \begin{pmatrix}
 1& x_1  & t_1\\
 x_1 & X_{11}& 0 \\
 t_1 & 0 & t_1\\
 \end{pmatrix}
  \begin{pmatrix} 1 & 0 & 1 & 0\\ 0 & 1 & -1 & 0 \\ 0 & 0 & -1 & 1\end{pmatrix}.
 \]
 Then $W\gesem 0$ if and only if
 \[
\begin{pmatrix} 1 & x_1 & t_1 \\
    x_1 & X_{11} & 0 \\
    t_1 & 0 & t_1
\end{pmatrix}
\gesem 0
\ \
\Leftrightarrow
\ \
 \begin{pmatrix}
  1-t_1 & x_1  \\
 x_1 & X_{11} \\
 \end{pmatrix} \gesem 0,
 \]
 which using $y_1=1-t_1$ is equivalent to $y_1 \ge 0$, $X_{11}\ge 0$, $y_1X_{11}\ge x_i^2$. The conditions of the theorem
 thus insure that $W\in\DNN$, where $t_1=1-y_1\ge0$, $s_1=1-t_1-x_1=y_1-x_1\ge0$, $Z_{11}=x_1-X_{11}\ge0$ and $S_{11}=1+X_{11}-2x_1-t_1=y_1+X_{11}-2x_1\ge0$.
 \end{proof}

Note that the characterization in Theorem \ref{thm:n=1} is
sometimes written in terms of the lower convex envelope rather than the
convex hull, in which case the condition $X_{11}\le x_1$ is omitted.

\section{\boldmath The disjunctive convex hull for $n=2$} \label{sec:disj}

In this section, we develop an explicit disjunctive formulation for the
convex hull $\Hcal$ when $n = 2$. As described in the Introduction,
we will use that fact that $\diag(Y) = y$ and that there is only one cross-term $y_1y_2$ to
write $(x,X,y,Y_{12})$ for points in $\Hcal$.

The representation for $\Hcal$ obtained in this section is based on the
four values of $y \in \{0,1\}^2 = \{0, e_1, e_2, e\}$. Specifically,
note that $\Hcal = \text{conv}(\Hcal_0 \cup \Hcal_{e_1} \cup \Hcal_{e_2}
\cup \Hcal_e)$, where for each fixed $y$,
\[
    \Hcal_y := \text{conv}\left\{ (x, xx^T, y, y_1 y_2) : 0 \le x \le y \right\}.
\]
Each such $\Hcal_y$ has a known representation. $\Hcal_0$ is just a
singleton, and for $y=e_1$ and $y=e_2$ representations based on $\PER$
are provided by Theorem \ref{thm:n=1}. For $y=e$, a representation is
given in \cite{Anstreicher.Burer.2010} as follows. Define
\[
    \RLT_x := \left\{ \begin{pmatrix} \lam & x\tran \\ x & X \end{pmatrix} \suchthat
        \begin{array}{l}
            \lam\ge 0, \ 0\le \diag(X) \le x \\
            \max\{0, x_1 + x_2 - \lam\} \le X_{12} \le \min\{x_1, x_2\}
        \end{array}
    \right\},
\]
which is the homogenization of those points $(x,X)$ satisfying the
standard RLT constraints associated with $0\le x \le e$. Then \cite{Anstreicher.Burer.2010}
\[
    \Hcal_e = \left\{ (x,X,y,Y_{12}) :
        \begin{pmatrix} 1 & x^T \\ x & X \end{pmatrix} \in \SDP \cap \RLT_x,
        \ y = e, \ Y_{12} = 1
    \right\},
\]
where PSD denotes the cone of positive semidefinite matrices. In the sequel we will also need
\[
    \RLT_y := \left\{ (y, Y_{12}) \in \Rbb^2 \times \Rbb :
        \begin{array}{l}
            \max\{0, y_1 + y_2 - 1\} \le Y_{12} \le \min\{y_1, y_2\}
        \end{array}
    \right\},
\]
which gives the convex hull of $(y, y_1y_2)$ over all four $y \in
\{0,1\}^2$. Note that $\RLT_y$ is a polytope, unlike $\PER$, $\RLT_x$ and $\SDP$, which
are convex cones.

In many applications, the product $y_1y_2$ is not of interest, so it
is also natural to consider the convex hull $\Hcal'$ that ignores
this product. Based on the known representations for $\Hcal_{e_1}$,
$\Hcal_{e_2}$ and $\Hcal_e$, $\Hcal^\prime$ is certainly contained in
the set of $(x,X,y)$ satisfying the constraints
\begin{eqnarray*}
    \begin{pmatrix} 1 & x\tran\\  x & X\end{pmatrix}
    &\in& \SDP \cap \RLT_x \\ 
    (x_j, X_{jj}, y_j ) &\in& \PER,\quad  y_j\le 1 \ \ \forall \ j = 1,2.\nonumber
\end{eqnarray*}
However it is easy to generate examples that satisfy these constraints
but are not in $\Hcal^\prime$. In the next theorem we will focus
on $\Hcal$, but we will return to a discussion of
$\Hcal^\prime$ in Section \ref{sec:conjecture}.

\begin{theorem} \label{thm:disj}
$\Hcal$ equals the projection onto $(x, X, y, Y_{12})$ of $(x, X, y, Y_{12},\alpha,\beta)$ satisfying the convex
constraints
\begin{subequations} \label{equ:disj}
\begin{align}
    &x \le y \\
    &\begin{pmatrix} Y_{12} & (x - \alpha)^T \\ x - \alpha & X - \Diag(\beta) \end{pmatrix}
    \in \SDP \cap \RLT_x \label{equ:sdprlt} \\
    &(\alpha_j, \beta_j, y_j - Y_{12}) \in \PER \ \ \forall \ j = 1,2 \label{equ:per} \\
    &(y, Y_{12}) \in \RLT_y
\end{align}
\end{subequations}
where $\alpha\in\Rbb^2$, $\beta \in \Rbb^2$ are auxiliary variables.
\end{theorem}

\begin{proof}
We first argue that (\ref{equ:disj}) is a relaxation of $\Hcal$ in
the lifted space that includes $\alpha$ and $\beta$. It suffices to
show that each ``rank-1'' solution $(x,xx^T,y,y_1y_2)$ for $y \in
\{0,1\}^2$ can be extended in $(\alpha,\beta)$ to a feasible solution
of (\ref{equ:disj}), and we handle the four cases for $y \in \{0,1\}^2$
separately. We clearly always have $x \le y$ and $(y, Y_{12}) \in \RLT_y$,
so it remains to check that (\ref{equ:sdprlt}) and (\ref{equ:per}) hold in
each case.

We introduce the notation
\[
    Z := \begin{pmatrix}
        Y_{12} & (x - \alpha)^T \\ x - \alpha & X - \Diag(\beta)
    \end{pmatrix}.
\]
First, let $y = 0 \Rightarrow x = 0$. Then $(x, xx^T, y, y_1 y_2) = (0,
0, 0, 0)$, and we choose $(\alpha, \beta) = (0, 0)$. Since all variables
are zero, it is straightforward to check that (\ref{equ:sdprlt}) and
(\ref{equ:per}) are satisfied. Second, let $y = e \Rightarrow 0 \le x
\le e$. Then $(x, xx^T, y, y_1 y_2) = (x, xx^T, e, 1)$, and we choose
$(\alpha, \beta) = (0, 0)$ for this case also, which yields $(\alpha_j,
\beta_j, y_j - Y_{12}) = (0,0,0) \in \PER$ for $j = 1,2$. Moreover,
\[
    Z = \begin{pmatrix} 1 & x^T \\ x & X \end{pmatrix}
    = \begin{pmatrix} 1 & x^T \\ x & xx^T \end{pmatrix} \in \SDP \cap
    \RLT_x,
\]
as desired.

Next we consider the case $y = e_1$, which implies $x_1 \le 1$ and $x_2
= 0$. Then $(x, xx^T, y, y_1 y_2) = (x_1 e_1, x_1^2 e_1 e_1^T, e_1, 0)$,
and we choose $(\alpha, \beta) = (x_1 e_1, x_1^2 e_1)$. Hence,
\[
    Z
    =\begin{pmatrix} 0 & (x - x_1 e_1)^T \\ x - x_1 e_1 & X - x_1^2 e_1 e_1^T \end{pmatrix}
    = 0 \in \SDP \cap \RLT_x,
\]
satisfying (\ref{equ:sdprlt}). Moreover, $(\alpha_1, \beta_1, y_1 - y_1
y_2) = (x_1, x_1^2, 1) \in \PER$ and $(\alpha_2, \beta_2, y_2 - y_1 y_2)
= (0, 0, 0) \in \PER$, so that (\ref{equ:per}) is satisfied. The final
case $y = e_2$ is similar. We have thus shown that (\ref{equ:disj}) is a
relaxation of $\Hcal$.

To complete the proof, we show the reverse containment, i.e., that any
$(x,X,y,Y_{12}, \alpha, \beta)$ satisfying (\ref{equ:disj}) is also a
member of $\Hcal$. Define the four scalars
\begin{equation}\label{equ:lam_y}
    \lam_0 := 1 - y_1 - y_2 + Y_{12},\ \ \lam_{e_1} := y_1 - Y_{12},\ \
    \lam_{e_2} := y_1 - Y_{12},\ \  \lam_e := Y_{12},
\end{equation}
and note that $(y,Y_{12}) \in \RLT_y$ implies $\lam_0 + \lam_{e_1}
+ \lam_{e_2} + \lam_e = 1$ with each term nonnegative, i.e.,
$(\lam_0,\lam_{e_1},\lam_{e_2},\lam_e)$ is a convex
combination. Next, letting $0/0 := 0$, define
\begin{alignat*}{2}
    Z_0 &:= \lam_0^{-1} \begin{pmatrix} \lam_0 & 0^T \\ 0 & 0 \end{pmatrix}
    &
    Z_{e_2} &:= \lam_{e_2}^{-1} \begin{pmatrix}
    \lam_{e_2} & \alpha_2 e_2^T \\
    \alpha_2 e_2 & \beta_2 e_2 e_2^T
    \end{pmatrix} \\
    Z_{e_1} &:= \lam_{e_1}^{-1} \begin{pmatrix}
    \lam_{e_1} & \alpha_1 e_1^T \\
    \alpha_1 e_1 & \beta_1 e_1 e_1^T
    \end{pmatrix} \quad\quad &
    Z_e &:= \lam_e^{-1} \begin{pmatrix} \lam_e & (x - \alpha)^T \\
        x - \alpha & X - \Diag(\beta)
    \end{pmatrix}.
\end{alignat*}
\noindent Note that $Z_y \in \Hcal_y$ for each $y \in \{0,1\}^2$;
for $y=e_1$ and $y=e_2$ we use the representation from
Theorem \ref{thm:n=1}, and for $y=e$ we use the result from
\cite{Anstreicher.Burer.2010} stated above this theorem. Hence, the
easily verified equations $(y, Y_{12}) = \lam_0 (0,0) + \lam_{e_1}
(e_1,0) + \lam_{e_2} (e_2,0) + \lam_e (e,1)$ and
\[
    \begin{pmatrix} 1 & x^T \\ x & X \end{pmatrix}
    = \lam_0 Z_0 + \lam_{e_1} Z_{e_1} + \lam_{e_2} Z_{e_2} + \lam_e Z_e,
    \]
establish that $(x,X,y,Y_{12}) \in \Hcal$.
\end{proof}

\section{\boldmath Eliminating $\beta$} \label{sec:beta}

System (\ref{equ:disj}) captures $\Hcal$ by projection from a lifted
space, which includes the additional variables $\alpha\in\Rbb^2$, $\beta
\in \Rbb^2$. In this section, we eliminate the $\beta$ variables from
\eqref{equ:disj}, but the price we pay is to replace the semidefinite
constraint in \eqref{equ:sdprlt} with PSD conditions on four matrices.
In Section \ref{sec:psd} we will will show that, in order to
obtain a characterization of $\Hcal$, it is in fact only necessary to
impose one of these four PSD conditions.

We begin by introducing some notation.
First, define the matrix function $M : \Rbb^2 \times \Sbb^2 \times \Rbb
\times \Rbb^2 \times \Rbb^2 \to \Sbb^3$ by
\begin{equation} \label{equ:M}
    M(\beta) := M(x,X,Y_{12},\alpha,\beta) :=
    \begin{pmatrix}
        Y_{12} & (x - \alpha)^T \\ x - \alpha & X - \Diag(\beta)
    \end{pmatrix}.
\end{equation}
The simplified notation $M(\beta)$ will be convenient because instances
of $M$ will only differ in the values of $\beta$; note also that $M$
does not depend on $y$. We also define four different functions
$\beta_{pq} : \Rbb^2 \times \Rbb \times \Rbb^2 \to \Rbb^2$ depending on
$(y,Y_{12},\alpha)$ for the indices $(p,q) \in \{1,2\}^2$, where $0/0 :=
0$:
%
\begin{align*}
    \beta_{11} &:= \beta_{11}(y,Y_{12},\alpha) := (X_{11}-x_1+\alpha_1,X_{22}-x_2+\alpha_2) \\
    \beta_{21} &:= \beta_{21}(y,Y_{12},\alpha) := \left( (y_1 - Y_{12})^{-1} \alpha_1^2, X_{22}-x_2+\alpha_2 \right) \\
    \beta_{12} &:= \beta_{12}(y,Y_{12},\alpha) := \left(X_{11}-x_1+\alpha_1, (y_2 - Y_{12})^{-1} \alpha_2^2 \right) \\
    \beta_{22} &:= \beta_{22}(y,Y_{12},\alpha) := \left( (y_1 - Y_{12})^{-1} \alpha_1^2, (y_2 - Y_{12})^{-1} \alpha_2^2 \right).
\end{align*}
%
As with $M(\beta)$, the shorter notation $\beta_{pq}$ will prove
more convenient. Note also that $p$ and $q$ are only index labels to
designate the four functions.  The result below replaces the PSD condition in \eqref{equ:sdprlt} with the four conditions
$M(\beta_{pq})\gesem 0$, $p,q\in \set{1,2}$.

\begin{theorem} \label{thm:nobeta}
$\Hcal$ equals the projection onto $(x, X, y, Y_{12})$ of $(x, X, y, Y_{12},\alpha)$ satisfying the convex
constraints
\begin{subequations} \label{equ:nobeta}
\begin{align}
    &\diag(X) \le x \le y \label{equ:nobeta:lin1} \\
    &\max\{0, x_1 - \alpha_1 + x_2 - \alpha_2 - Y_{12} \} \le X_{12} \le
    \min\{x_1 - \alpha_1, x_2 - \alpha_2 \} \label{equ:nobeta:lin2} \\
    &0 \le \alpha_j \le y_j - Y_{12} \ \ \forall \ j = 1,2 \label{equ:nobeta:lin3} \\
    &(y, Y_{12}) \in \RLT_y \label{equ:nobeta:lin4} \\
    &M(\beta_{11}) \succeq 0 \label{equ:3x3} \\
    &M(\beta_{12}) \succeq 0 \label{equ:4x4first} \\
    &M(\beta_{21}) \succeq 0 \label{equ:4x4second} \\
    &M(\beta_{22}) \succeq 0 \label{equ:5x5}.
\end{align}
\end{subequations}
\end{theorem}

\begin{proof}

The proof is based on reformulating (\ref{equ:disj}), which using
$M(\beta)$ can be restated as
%
\begin{align*}
    &x \le y \\
    &M(\beta) \in \SDP \cap \RLT_x \\ 
    &(\alpha_j, \beta_j, y_j - Y_{12}) \in \PER \ \ \forall \ j = 1,2 \\
    &(y, Y_{12}) \in \RLT_y.
\end{align*}
In particular, considering $(x,X,y,Y_{12},\alpha)$ fixed,
the above system includes four linear conditions on $\beta$:
\[
    \beta_j \ge
    \max \left\{ (y_j - Y_{12})^{-1} \alpha_j^2, X_{jj} - x_j + \alpha_j \right\}
    \ \ \ \ \forall \ \ j = 1,2.
\]
Moreover, since decreasing $\beta_1$ and $\beta_2$ while holding all
other variables constant does not violate $M(\beta) \succeq 0$, we may
define $\beta_1$ and $\beta_2$ by
\[
    \beta_j(x,X,y,Y_{12},\alpha) :=
        \max \left\{ (y_j - Y_{12})^{-1} \alpha_j^2, X_{jj} - x_j + \alpha_j \right\}
        \ \ \ \ \forall \ \ j = 1,2
\]
without affecting the projection onto $(x,X,y,Y_{12})$. It follows that values
$(x, X, y, Y_{12},\alpha)$, which are feasible for \eqref{equ:nobeta:lin1}--\eqref{equ:nobeta:lin4},
are feasible for the constraints \eqref{equ:disj} if and only if $M(\beta_{pq})\gesem 0$,
$(p,q)\in\set{1,2}^2$.
\end{proof}

\leaveout{
In other words,
after simplification, $\Hcal$ equals the projection of
\begin{subequations} \label{equ:disj3}
\begin{align}
    &\diag(X) \le x \le y \\
    &\max\{0, x_1 - \alpha_1 + x_2 - \alpha_2 - Y_{12} \} \le X_{12} \le
    \min\{x_1 - \alpha_1, x_2 - \alpha_2 \} \\
    &0 \le \alpha_j \le y_j - Y_{12} \ \ \forall \ j = 1,2 \\
    &(y, Y_{12}) \in \RLT_y \\
    &M(\beta(x,X,y,Y_{12},\alpha)) \succeq 0.
\end{align}
\end{subequations}
Observing that $\beta(x,X,y,Y_{12},\alpha)$ is defined piecewise
by four functions, it follows that the matrix function
$M(\beta(x,X,y,Y_{12},\alpha))$ is also defined by four pieces. One can
easily check that these four pieces are precisely the four functions
$M(\beta_{pq})$. Moreover, by again exploiting the fact that
decreasing $\beta_1$ and $\beta_2$ preserves positive semidefiniteness,
we have
\[
    M(\beta(x,X,y,Y_{12},\alpha)) \succeq 0
    \ \ \Longleftrightarrow \ \
    M(\beta_{pq}) \succeq 0 \ \ \forall \ (p,q) \in \{1,2\}^2.
\]
So (\ref{equ:nobeta}) is equivalent to (\ref{equ:disj3}). This proves
the first part of the theorem.

The second part of the theorem follows by reformulating the PSD
conditions $M(\beta_{pq}) \succeq 0$ using the Schur complement
theorem.}

In Section \ref{sec:psd}, we will show that in order to obtain an exact representation of
$\Hcal$ only the condition $M(\beta_{22}) \succeq 0$ is required.  For
clarity in the exposition it is helpful to write out the conditions
$M(\beta_{pq}) \succeq 0$ explicitly. In particular,
 (\ref{equ:3x3}) can be written
\begin{equation} \label{equ:3x3'} \tag{\ref{equ:3x3}$'$}
    \begin{pmatrix} Y_{12} & x_1 - \alpha_1 & x_2 - \alpha_2 \\
        x_1 - \alpha_1 & x_1 - \alpha_1 & X_{12} \\
        x_2 - \alpha_2 & X_{12} & x_2 - \alpha_2
    \end{pmatrix} \succeq 0.
\end{equation}
In the remaining cases we can utilize the well-known Schur complement condition
to conclude that
(\ref{equ:4x4first}) is equivalent to
\begin{equation} \label{equ:4x4first'} \tag{\ref{equ:4x4first}$'$}
    \begin{pmatrix}
        y_1 - Y_{12} & 0 & \alpha_1 & 0 \\
        0 & Y_{12} & x_1 - \alpha_1 & x_2 - \alpha_2 \\
        \alpha_1 & x_1 - \alpha_1 & X_{11} & X_{12} \\
        0 & x_2 - \alpha_2 & X_{12} & x_2 - \alpha_2
    \end{pmatrix} \succeq 0,
\end{equation}
(\ref{equ:4x4second}) is equivalent to
\begin{equation} \label{equ:4x4second'} \tag{\ref{equ:4x4second}$'$}
    \begin{pmatrix}
        y_2 - Y_{12} & 0 & 0 & \alpha_2 \\
        0 & Y_{12} & x_1 - \alpha_1 & x_2 - \alpha_2 \\
        0 & x_1 - \alpha_1 & x_1 - \alpha_1 & X_{12} \\
        \alpha_2 & x_2 - \alpha_2 & X_{12} & X_{22}
    \end{pmatrix} \succeq 0,
\end{equation}
and (\ref{equ:5x5}) is equivalent to
\begin{equation} \label{equ:5x5'} \tag{\ref{equ:5x5}$'$}
    \begin{pmatrix}
        y_1 - Y_{12} & 0 & 0 & \alpha_1 & 0 \\
        0 & y_2 - Y_{12} & 0 & 0 & \alpha_2 \\
        0 & 0 & Y_{12} & x_1 - \alpha_1 & x_2 - \alpha_2 \\
        \alpha_1 & 0 & x_1 - \alpha_1 & X_{11} & X_{12} \\
        0 & \alpha_2 & x_2 - \alpha_2 & X_{12} & X_{22}
    \end{pmatrix} \succeq 0.
\end{equation}

\noindent In the statement of results in the sequel we will always refer to the conditions \eqref{equ:3x3}--\eqref{equ:5x5}, but
these statements may be easier to understand if the reader refers to \eqref{equ:3x3'}--\eqref{equ:5x5'}.

\section{Reducing to a single semidefinite condition} \label{sec:psd}

Theorem \ref{thm:nobeta} establishes that $\Hcal$ is described in part by
the four PSD conditions (\ref{equ:3x3})--(\ref{equ:5x5})---one of size
$3 \times 3$, two of size $4 \times 4$, and one of size $5 \times 5$.
In this section, we show that Theorem \ref{thm:nobeta} holds even if
(\ref{equ:3x3})--(\ref{equ:4x4second}) are not enforced.  We show
this in several steps. First, we prove that (\ref{equ:3x3}) is redundant.

\subsection{Condition (\ref{equ:3x3}) is redundant}

\begin{lemma} \label{pro:3x3redundant}
If $(x,X,y,Y_{12},\alpha)$ satisfies
(\ref{equ:nobeta:lin1})--(\ref{equ:nobeta:lin4}), then it satisfies
(\ref{equ:3x3}).
\end{lemma}

\begin{proof}
Consider the linear conditions
(\ref{equ:nobeta:lin1})--(\ref{equ:nobeta:lin4}) of (\ref{equ:nobeta}).
In terms of the remaining variables, the constraints on $X_{12}$ are simple
bounds:
\[
    l := \max\{0, x_1 - \alpha_1 + x_2 - \alpha_2 - Y_{12} \}
    \le X_{12} \le
    \min\{ x_1 - \alpha_1, x_2 - \alpha_2 \} =: u.
\]
We claim that (\ref{equ:3x3}) is satisfed at both endpoints
$X_{12} = l$ and $X_{12} = u$, which will prove the
theorem since the determinant of every principal submatrix of
$M(\beta_{11})$ that includes $X_{12}$ is a concave
quadratic function of $X_{12}$.

So we need $M(\beta_{11}) \succeq 0$ at both $X_{12} = l$ and
$X_{12} = u$, i.e.,
\[
    \begin{pmatrix} Y_{12} & x_1 - \alpha_1 & x_2 - \alpha_2 \\
        x_1 - \alpha_1 & x_1 - \alpha_1 & l \\
        x_2 - \alpha_2 & l & x_2 - \alpha_2
    \end{pmatrix} \succeq 0\quad
\mbox{and}\quad
    \begin{pmatrix} Y_{12} & x_1 - \alpha_1 & x_2 - \alpha_2 \\
        x_1 - \alpha_1 & x_1 - \alpha_1 & u \\
        x_2 - \alpha_2 & u & x_2 - \alpha_2
    \end{pmatrix} \succeq 0.
\]
The two matrices above share several properties necessary for positive
semidefiniteness. Both have nonnegative diagonals, and all $2 \times 2$
principal minors are nonnegative:

\begin{itemize}

\item For each, the $\{1,2\}$ principal minor is nonnegative if
and only if $Y_{12} (x_1 - \alpha_1) - (x_1 - \alpha_1)^2 \ge 0$. This
follows from (\ref{equ:nobeta:lin2}):
\begin{equation} \label{equ:local1}
    Y_{12} \ge (x_1 - \alpha_1) + (x_2 - \alpha_2 - X_{12}) \ge (x_1 - \alpha_1) + 0
    = x_1 - \alpha_1,
\end{equation}
which implies $Y_{12} (x_1 - \alpha_1) \ge (x_1 - \alpha_1)^2$.

\item For each, the $\{1,3\}$ principal minor is similarly
nonnegative.

\item The respective $\{2,3\}$ minors are nonnegative if $(x_1 -
\alpha)(x_2 - \alpha_2) - l^2 \ge 0$ and $(x_1 - \alpha_1)(x_2 -
\alpha_2) - u^2 \ge 0$, which hold because $0 \le l \le u \le x_1 -
\alpha_1$ and $0 \le l \le u \le x_2 - \alpha_2$.

\end{itemize}

\noindent It remains to show that the both determinants  of both matrices are nonnegative.
Let us first examine the case for $X_{12}=l$, which itself breaks into two
subcases: (i) $x_1 - \alpha_1 + x_2 - \alpha_2 - Y_{12} \le 0 = l$; (ii)
$0 \le x_1 - \alpha_1 + x_2 - \alpha_2 - Y_{12} = l$. For subcase (i),
the determinant equals $(x_1 - \alpha_1)(x_2 - \alpha_2)(Y_{12} - x_1 +
\alpha_1 - x_2 + \alpha_2)$, which is the product of three nonnegative
terms. For subcase (ii), the determinant equals
\[
    (Y_{12} - x_2 + \alpha_2)(Y_{12} - x_1 + \alpha_1)(x_1 - \alpha_1 + x_2 - \alpha_2 - Y_{12})
\]
which is also the product of three nonnegative terms; in particular, see
(\ref{equ:local1}).
The case for $X_{12}=u$ similarly breaks down into two subcases, which mirror
(i) and (ii) above.
\end{proof}

\subsection{\boldmath Reduction to $\alpha_1 = 0$}\label{sub:alpha1=0}

In order to prove that Theorem
\ref{thm:nobeta} holds even without (\ref{equ:4x4first}) and
(\ref{equ:4x4second}), we will first reduce to the case $\alpha_1
= 0$. In fact if \eqref{equ:nobeta:lin1}--\eqref{equ:nobeta:lin4} and
\eqref{equ:5x5} hold, then
at most one of (\ref{equ:4x4first}) and (\ref{equ:4x4second}) can be
violated. This is because, if both were violated, then we would have
$X_{11} - \alpha_1^2 / (y_1 - Y_{12}) > x_1 - \alpha_1$ and $X_{22} -
\alpha_2^2 / (y_2 - Y_{12}) > x_2 - \alpha_2$; otherwise, by comparing
diagonal elements (\ref{equ:5x5}) would not hold. However, these two
strict inequalities then imply that (\ref{equ:3x3}) $\Rightarrow$
(\ref{equ:4x4first})--(\ref{equ:5x5}), which is a contradiction.
So we assume without loss of generality that (\ref{equ:4x4first})
is violated while (\ref{equ:4x4second}) holds, and use the following terminology regarding system
(\ref{equ:nobeta}): we say that a point $(x,X,y,Y_{12},\alpha)$
{\em lacks only} (\ref{equ:4x4first}) when the point satisfies all conditions in
(\ref{equ:nobeta}) except that it violates (\ref{equ:4x4first}).

\begin{lemma} \label{lem:lacks}
Suppose that $(x,X,y,Y_{12},\alpha)$ lacks only (\ref{equ:4x4first}), and suppose $\alpha_1>0$. Then $y_1 - Y_{12} > 0$ and
$(\bar x, \bar X, y, Y_{12}, \bar \alpha)$ lacks only (\ref{equ:4x4first}),
where
\[
    \bar x := {x_1 - \alpha_1 \choose x_2},
    \ \ \ \ \
    \bar X := \begin{pmatrix} X_{11} - \alpha_1^2 / (y_1 - Y_{12}) & X_{12} \\
    X_{12} & X_{22} \end{pmatrix},
    \ \ \ \ \
    \bar\alpha := {0 \choose \alpha_2}.
\]
\end{lemma}

\begin{proof}
If $\alpha_1>0$ then \eqref{equ:5x5} implies that $y_1-Y_{12}>0$.
For notational convenience, define $v:=
(x,X,y,Y_{12},\alpha)$ and $\bar v := (\bar x, \bar X, y, Y_{12}, \bar
\alpha)$. We need to check that $\bar v$ satisfies all conditions
in (\ref{equ:nobeta}) except (\ref{equ:4x4first}). Since only $\bar
x_1$, $\bar X_{11}$, and $\bar\alpha_1$ differ between $v$ and $\bar
v$, and since $\bar x_1 - \bar\alpha_1 = x_1 - \alpha_1$, we need to
verify $\bar X_{11} \le \bar x_1 \le y_1$, $0 \le \bar\alpha_1 \le
y_1 - Y_{12}$, and (\ref{equ:5x5}) at $\bar v$, and we need to show
(\ref{equ:4x4first}) does {\em not\/} hold at $\bar v$. Clearly $0 \le
\bar\alpha_1 \le y_1 - Y_{12}$ because $\bar\alpha_1 = 0$, and $\bar
x_1 \le x_1 \le y_1$.

With $\bar\alpha_1 = 0$ and $\bar x_1 = x_1 - \alpha_1$, conditions
(\ref{equ:3x3}) and (\ref{equ:4x4first}) at $\bar v$ are respectively
equivalent to
\[
    \begin{pmatrix}
        Y_{12} & \bar x_1 & x_2 - \alpha_2 \\
        \bar x_1 & \bar x_1 & X_{12} \\
        x_2 - \alpha_2 & X_{12} & x_2 - \alpha_2
    \end{pmatrix}
    =
    \begin{pmatrix}
        Y_{12} & x_1 - \alpha_1 & x_2 - \alpha_2 \\
        x_1 - \alpha_1 & x_1 - \alpha_1 & X_{12} \\
        x_2 - \alpha_2 & X_{12} & x_2 - \alpha_2
    \end{pmatrix}
    \succeq 0,
\]
and
\[
    \begin{pmatrix}
        Y_{12} & \bar x_1 & x_2 - \alpha_2 \\
        \bar x_1 & \bar X_{11} & X_{12} \\
        x_2 - \alpha_2 & X_{12} & x_2 - \alpha_2
    \end{pmatrix}
    =
    \begin{pmatrix}
        Y_{12} & x_1 - \alpha_1 & x_2 - \alpha_2 \\
        x_1 - \alpha_1 & X_{11} - \alpha_1^2 / (y_1 - Y_{12}) & X_{12} \\
        x_2 - \alpha_2 & X_{12} & x_2 - \alpha_2
    \end{pmatrix}
    \succeq 0.
\]

These conditions both match the conditions of (\ref{equ:3x3}) and
(\ref{equ:4x4first}) at $v$, showing  that (\ref{equ:3x3})
holds at $v$ if and only if (\ref{equ:3x3}) holds at $\bar v$, and
similarly for (\ref{equ:4x4first}). In particular, this implies $\bar
v$ does not satisfy (\ref{equ:4x4first}), as desired. In addition, we
conclude $\bar X_{11} \le \bar x_1$ because, if $\bar X_{11}$ were
greater than $\bar x_1$, then (\ref{equ:3x3}) holding at $v$ would
imply (\ref{equ:4x4first}) holds at $v$ by just comparing the diagonal
elements above, but this would violate our assumptions.

Finally, using again the
relationship between $\bar v$ and $v$, (\ref{equ:5x5})
holds at $\bar v$ if and only if
\[
\begin{pmatrix}
    y_2 - Y_{12} & 0 & 0 & \alpha_2 \\
    0 & Y_{12} & x_1 - \alpha_1 & x_2 - \alpha_2 \\
    0 & x_1 - \alpha_1 & X_{11} - \alpha_1^2 / (y_1 - Y_{12}) & X_{12} \\
    \alpha_2 & x_2 - \alpha_2 & X_{12} & X_{22}
\end{pmatrix} \succeq 0,
\]
which is true by applying the Schur complement, using the fact that (\ref{equ:5x5}) holds at $v$.
\end{proof}

\subsection{\boldmath Characterizing (\ref{equ:4x4first}) and (\ref{equ:5x5}) in
terms of $\alpha_2$}

Given $(x,X,y,Y_{12},\alpha)$ with $\alpha_1=0$ that lacks only
\eqref{equ:4x4first}, in Section \ref{sub:adjust} our goal will be to
modify $\alpha_2$ to a new value $\hat\alpha_2$ so as to satisfy all
the constraints of \eqref{equ:nobeta}. To facilitate this analysis,
we now carefully examine how conditions (\ref{equ:4x4first}) and
(\ref{equ:5x5}) depend on $\alpha_2$.

Because $y_1-Y_{12} \ge 0$ and $\alpha_1=0$, (\ref{equ:4x4first}) is
equivalent to
\begin{equation} \label{equ:4f}
V := \begin{pmatrix}
Y_{12} & x_1 &x_2-\alpha_2 \cr
x_1& X_{11} & X_{12} \cr
x_2-\alpha_2& X_{12}& x_2-\alpha_2
\end{pmatrix} \gesem 0.
\end{equation}
Now letting $\xbar_2:=x_2-\alpha_2$, we have
$
    \det(V) =
    -X_{11}\xbar_2^2 + (2X_{12}x_1 + Y_{12} X_{11}-x_1^2)\xbar_2 - Y_{12} X_{12}^2.
$
As a function of $\bar x_2$, this is a strictly concave quadratic assuming that
$X_{11} > 0$. Moreover, the discriminant for this quadratic is
\begin{eqnarray*}
& &(Y_{12} X_{11} - x_1^2 +2x_1X_{12})^2-4Y_{12} X_{11}X_{12}^2\\
&=& (Y_{12} X_{11} - x_1^2)^2 + 4x_1X_{12}(Y_{12}X_{11}-x_1^2)+4x_1^2X_{12}^2-4Y_{12} X_{11}X_{12}^2\\
&=&  (Y_{12} X_{11} - x_1^2)^2 +4x_1^2X_{12}(X_{12}-x_1) + 4Y_{12} X_{11}X_{12}(x_1-X_{12})\\
&=&(Y_{12}X_{11}-x_1^2)^2+ 4X_{12}(x_1-X_{12})(Y_{12}X_{11}-x_1^2)\\
&=&\theta(\theta+4X_{12}(x_1-X_{12})),
\end{eqnarray*}
where $\theta:=Y_{12}X_{11}-x_1^2\ge 0$. It follows that $\det(V) \ge 0$
if and only if $\xbar_2$ is contained in the interval bounded by the roots
\[
\frac{X_{12}x_1}{X_{11}} + \frac{\theta \pm \sqrt{\theta(\theta+4X_{12}(x_1-X_{12}))}}{2X_{11}},
\]
or equivalently, if and only if $\alpha_2 \in [\alpha_2^-,\alpha_2^+]$, where
\begin{subequations} \label{equ:alpha+-def}
\begin{eqnarray}
\alpha_2^-&:=& x_2-\frac{X_{12}x_1}{X_{11}} - \frac{\theta + \sqrt{\theta(\theta+4X_{12}(x_1-X_{12}))}}{2X_{11}}\ \le\  x_2-\frac{X_{12}x_1}{X_{11}}
-\frac{\theta}{X_{11}}\label{eq:alpha2-}\\
\alpha_2^+&:=& x_2-\frac{X_{12}x_1}{X_{11}} - \frac{\theta - \sqrt{\theta(\theta+4X_{12}(x_1-X_{12}))}}{2X_{11}}\ \ge\  x_2-\frac{X_{12}x_1}{X_{11}}. \label{eq:alpha2+}
\end{eqnarray}
\end{subequations}
\leaveout{The inequalities in (\ref{eq:alpha2-}) and (\ref{eq:alpha2+}) follow,
respectively, from the inequalities
\[
    -\frac{ \sqrt{\theta(\theta + 4 X_{12} (x_1 - X_{12}))} }{ 2 X_{11} } \le 0
    \ \ \ \ \ \text{and} \ \ \ \ \
    - \frac{\theta - \sqrt{\theta(\theta+4X_{12}(x_1-X_{12}))}}{2X_{11}} \ge 0.
\]}

From the above, if $(x,X,y,Y_{12},\alpha)$ lacks only (\ref{equ:4x4first}), where
$\alpha_1=0$ and $X_{11} > 0$ then to have $(x,X,y,Y_{12},\alphahat)$ satisfy (\ref{equ:4x4first})
with $\alphahat_1=0$ we certainly require that $\alphahat_2\in [\alpha_2^-,\alpha_2^+]$.  In the next lemma
we show that in fact this condition is necessary and sufficient.

\begin{lemma} \label{lem:alphabounds}
Suppose $(x,X,y,Y_{12},\alpha)$ lacks only (\ref{equ:4x4first}), where
$\alpha_1=0$, and let $\alphahat := (0,\alphahat_2)$. Then $X_{11}>0$, $y_2-Y_{12}>0$, and $(x,X,y,Y_{12},\alphahat)$ satisfies (\ref{equ:4x4first}) if
and only if $\alphahat_2\in [\alpha_2^-,\alpha_2^+]$.
\end{lemma}

\begin{proof}
Note that if
$(x,X,y,Y_{12},\alpha)$ with $\alpha_1=0$ satisfies \eqref{equ:5x5}, then
$X_{11}=0$ implies that $x_1=X_{12}=0$.  In this case \eqref{equ:4x4first} follows immediately from \eqref{equ:nobeta:lin2}.
 In addition,
if $y_2-Y_{12}=0$ then \eqref{equ:5x5} implies that $\alpha_2=0$, in which case \eqref{equ:4x4first} would follow immediately from
$X_{22}\le x_2$.  Thus if $(x,X,y,Y_{12},\alpha)$ with $\alpha_1=0$ lacks only \eqref{equ:4x4first} we must have $X_{11}>0$ and $y_2-Y_{12}>0$.

We consider $V$ defined in (\ref{equ:4f}) with $\alphahat_2$
substituted for $\alpha_2$; we wish to show $V \succeq 0$ if and only
if $\alphahat_2\in [\alpha_2^-,\alpha_2^+]$. As discussed before the
lemma, $\det(V) \ge 0$ for such $\hat\alpha_2$, but it could happen
that $V \not\succeq 0$ even when $\det(V)\ge 0$. Note that, since
$(x,X,y,Y_{12},\alpha)$ satisfies (\ref{equ:5x5}) by assumption, then by
the eigenvalue interlacing theorem (see, for example, Theorem 4.3.8 of
Horn and Johnson \cite{Horn.Johnson.1985}), $V$ has at most one negative
eigenvalue.

We consider two cases based on whether $\theta \ge 0$ is positive
or zero. If $\theta>0$, then by the determinant and discriminant formulas above
we have $\det(V) > 0 \Rightarrow V \succ 0$ for $\alphahat_2\in
(\alpha_2^-,\alpha_2^+)$, and $V \succeq 0$ with $\det(V) = 0$ when
$\alphahat_2=\alpha_2^-$ or $\alphahat_2=\alpha_2^+$. The latter
follows, for example, by continuity of the determinants of all principal
submatrices. On the other hand, if $\theta=0$, then: $\alpha_2^- = \alpha_2^+
= x_2-X_{12}x_1/X_{11}$; $\det(V) = 0$ when $\hat\alpha_2=
x_2-X_{12}x_1/X_{11}$; and $\det(V) < 0$ for any other value of
$\alphahat_2$. Focusing then on $\alphahat_2=x_2-X_{12}x_1/X_{11}$,
we have
\[
V = \begin{pmatrix}
Y_{12} & x_1 & X_{12}x_1/X_{11} \\
x_1 & X_{11} & X_{12} \\
X_{12}x_1/X_{11} & X_{12} & X_{12}x_1/X_{11}
\end{pmatrix}.
\]
In this case $\diag(V) \ge 0$ and $\det(V) = 0$, so to demonstrate $V
\succeq 0$, we need to show that the $2\times 2$ principal submatrices are
positive semidefinite or equivalently have nonnegative determinants.
The $\set{1,2}$ submatrix is positive semidefinite since (\ref{equ:5x5})
is satisfied; the determinant of the $\set{1,3}$ submatrix is
nonnegative because $Y_{12}X_{11} \ge x_1^2 \ge X_{12}x_1$; and the
determinant of the $\set{2,3}$ submatrix is nonnegative because $x_1\ge
X_{12}$.
\end{proof}

It will also be important that we understand how (\ref{equ:5x5}) depends
on $\alpha_2$. When $\alpha_1 = 0$ and
$(x,X,y,Y_{12},\alpha)$ satisfies (\ref{equ:5x5}), we certainly have
\begin{equation}\label{equ:2gdet}
\left| \begin{matrix}
y_2-Y_{12} & 0 & 0 & \alpha_2 \\
0 & Y_{12} & x_1 &x_2-\alpha_2 \\
0 & x_1& X_{11} & X_{12} \\
\alpha_2 & x_2-\alpha_2& X_{12}& X_{22}
\end{matrix} \right| \ge 0.
\end{equation}
Assuming that  $X_{11}>0$ and $y_2 - Y_{12} > 0$, the left side of \eqref{equ:2gdet} is a strictly
concave quadratic function of $\alpha_2$, and it is straightforward to
compute that the maximizer of this determinant is
\begin{equation}\label{eq:alpha2*}
    \alpha_2^* := \frac{(y_2-Y_{12})(x_2 X_{11} - x_1 X_{12})}{y_2X_{11} - x_1^2}
    = \left(x_2-\frac {X_{12}x_1}{X_{11}}\right) \frac{y_2-Y_{12}}{y_2-x_1^2/X_{11}}
    \le x_2-\frac {X_{12}x_1}{X_{11}}.
\end{equation}
In \eqref{eq:alpha2*} the denominator $y_2X_{11} - x_1^2$ is strictly
positive since $Y_{12}X_{11} \ge x_1^2$ and $y_2 > Y_{12}$, and then the
inequality follows from the fact that $Y_{12}X_{11} \ge x_1^2$.

Finally, for $\alpha_1=0$ the lemma below considers conditions under which
(\ref{equ:4x4first}) $\Rightarrow$ (\ref{equ:5x5}), and
(\ref{equ:5x5}) $\Rightarrow$ (\ref{equ:4x4first}).

\begin{lemma} \label{lem:psdimpliespsd}
Let $(x,X,y,Y_{12},\alpha)$ be given with $\alpha_1 = 0$, $y_2-Y_{12}>0$ and $0 \le x_2
- X_{22} \le \tfrac14 (y_2 - Y_{12})$. Define $\rho := \sqrt{1 - 4 (x_2
- X_{22}) / (y_2 - Y_{12})} \le 1$. Also define
$$\lam^- := \tfrac12(1 - \rho) (y_2 - Y_{12}) \le \tfrac12(1 + \rho)
(y_2 - Y_{12}) =: \lam^+.$$ Then $\lam^- \le \alpha_2 \le \lam^+$ ensures
(\ref{equ:4x4first}) $\Rightarrow$ (\ref{equ:5x5}), and $\alpha_2 \le
\lam^-$ or $\lam^+ \le \alpha_2$ ensures (\ref{equ:5x5}) $\Rightarrow$
(\ref{equ:4x4first}).
\end{lemma}

\begin{proof}
By exploiting $\alpha_1 = 0$, using the Schur complement theorem, and
comparing diagonal elements, we see that: (i) (\ref{equ:4x4first})
$\Rightarrow$ (\ref{equ:5x5}) is ensured when $x_2 - \alpha_2 \le X_{22}
- \alpha_2^2 / (y_2 - Y_{12})$; and (ii) (\ref{equ:5x5}) $\Rightarrow$
(\ref{equ:4x4first}) is ensured when the reverse inequality $x_2 - \alpha_2 \ge X_{22} -
\alpha_2^2 / (y_2 - Y_{12})$ holds. Note that $\lam^-$ and $\lam^+$ are the
roots of the quadratic equation $x_2 - \alpha_2 = X_{22} - \alpha_2^2 /
(y_2 - Y_{12})$ in $\alpha_2$. In particular, the assumption $0 \le x_2
- X_{22} \le \tfrac14 (y_2 - Y_{12})$ guarantees that the discriminant
is nonnegative and that $x_2 - \alpha_2 \le X_{22} - \alpha_2^2 / (y_2 -
Y_{12})$ is satisfied at the midpoint $\tfrac12(y_2 - Y_{12})$ of
$\lam^-$ and $\lam^+$. Then the final statement of the lemma is just the
restatement of (i) and (ii).
\end{proof}

\subsection{\boldmath Adjusting $\alpha_2$ when $\alpha_1 = 0$} \label{sub:adjust}

Assume that $(x,X,y,Y_{12},\alpha)$ lacks only (\ref{equ:4x4first}) with
$\alpha_1=0$. Then by Lemma \ref{lem:alphabounds} either $\alpha_2 <
\alpha_2^-$ or $\alpha_2>\alpha_2^+$; see (\ref{equ:alpha+-def}) for the
definitions of $\alpha_2^-$ and $\alpha_2^+$. The next two lemmas
show that $(x,X,y,Y_{12},\alphahat)$ then satisfies (\ref{equ:nobeta}), where
in the first case $\alphahat = (0,\alpha_2^-)$ and in the second case
$\alphahat = (0,\alpha_2^+)$.

\begin{lemma}\label{lem:increase_alpha2}
Assume that $(x,X,y,Y_{12},\alpha)$ lacks only (\ref{equ:4x4first}) with
$\alpha_1=0$, and $\alpha_2 <\alpha_2^-$. Then
$(x,X,y,Y_{12},\alphahat)$ satisfies (\ref{equ:nobeta}) with $\alphahat
= (0,\alpha_2^-)$.
\end{lemma}

\begin{proof}
From Lemma \ref{lem:alphabounds} we know that $X_{11}>0$, $y_2-Y_{12}>0$ and
$(x,X,y,Y_{12},\alphahat)$
satisfies (\ref{equ:4x4first}). Since
(\ref{equ:nobeta:lin1})--(\ref{equ:nobeta:lin4}) $\Rightarrow$
(\ref{equ:3x3}) by Proposition \ref{pro:3x3redundant}
and (\ref{equ:5x5}) $\Rightarrow$ (\ref{equ:4x4second})
when $\alpha_1 = 0$ by inspection, we need to establish
just (\ref{equ:nobeta:lin1})--(\ref{equ:nobeta:lin4}) and
(\ref{equ:5x5}). Since $(x,X,y,Y_{12},\alpha)$ satisfies
(\ref{equ:nobeta:lin1})--(\ref{equ:nobeta:lin4}) and we have increased
$\alpha_2$ to $\alpha_2^-$ to form $\alphahat$, we need only show
$\alpha_2^- \le x_2-X_{12}$ and $\alpha_2^- \le y_2-Y_{12}$ to establish that
(\ref{equ:nobeta:lin1})--(\ref{equ:nobeta:lin4}) hold for $(x,X,y,Y_{12},\alphahat)$.
 In fact, we will
show $\alpha_2^- \le x_2 - X_{12}$ as well as the stronger inequality
$\alpha_2^- \le \lam^+$, where $\lam^+ = \tfrac12 (1 + \rho) (y_2 - Y_{12})$
and $0 \le \rho \le 1$ are defined in Lemma \ref{lem:psdimpliespsd}.
Indeed, the conditions of Lemma \ref{lem:psdimpliespsd} hold
here because, as (\ref{equ:5x5}) is satisfied but
(\ref{equ:4x4first}) is violated at $\alpha_2$, we have $x_2 - \alpha_2 \le X_{22}
- \alpha_2^2 / (y_2 - Y_{12})$, which ensures $0 \le x_2 - X_{22} \le
\tfrac14(y_2 - Y_{12})$ and $\alpha_2 \le \lam^+$. Hence, proving
$\alpha_2^- \le \lam^+$ will ensure (\ref{equ:4x4first})
$\Rightarrow$ (\ref{equ:5x5}).

To prove $\alpha_2^- \le x_2 - X_{12}$, we note that \eqref{eq:alpha2-} and $x_1\ge
X_{11}$ imply
\[
\alpha_2^-\le x_2- \frac{X_{12}x_1}{X_{11}} \le x_2-X_{12}.
\]
Next, to prove $\alpha_2^- \le \lam^+$, assume for contradiction that
$\alpha_2\le \lam^+ <\alpha_2^-$. Consider $\alpha_2^*$ as defined
in (\ref{eq:alpha2*}). We claim $\lam^+ < \alpha_2^*$, which from
\eqref{eq:alpha2*} is equivalent to
 \[
     x_2-\frac{X_{12}x_1}{X_{11}} > \tfrac12 (1 + \rho)\left(y_2-\frac{x_1^2}{X_{11}}\right).
 \]
From \eqref{eq:alpha2-}, the definition of $\theta$, and the assumption
that $\lam^+ <\alpha_2^-$, we then have
\begin{align*}
    x_2-\frac{X_{12}x_1}{X_{11}}
    &\ge \alpha_2^- + \frac{\theta}{X_{11}} \\
    &> \tfrac12(1 + \rho)(y_2-Y_{12}) + \left(Y_{12} -\frac{x_1^2}{X_{11}}\right) \\
    &\ge \tfrac12(1 + \rho)(y_2-Y_{12}) + \tfrac12(1 + \rho)\left(Y_{12} -\frac{x_1^2}{X_{11}}\right) \\
    &= \tfrac12(1 + \rho)\left(y_2-\frac{x_1^2}{X_{11}}\right),
\end{align*}
as required. Since \eqref{equ:2gdet} holds at $\alpha_2 \le \lam^+$ and
$\alpha_2^* > \lam^+$, the determinant in \eqref{equ:2gdet} must be strictly positive
at $\lam^+$; recall that this determinant is a strictly concave function of $\alpha_2$.
Then (\ref{equ:5x5}) holds with $\alpha_2$ replaced by
$\lam^+$, since eigenvalue interlacing implies that the matrix in \eqref{equ:5x5} can
have at most one negative eigenvalue as $\alpha_2$ is varied. However Lemma \ref{lem:psdimpliespsd} then implies that (\ref{equ:4x4first}) also then holds with $\alpha_2$ replaced by
 $\lam^+$, and therefore $\alpha_2^- \le \lam^+$ from Lemma \ref{lem:alphabounds}.
This is the desired
contradiction of $\lam^+ < \alpha_2^-$.  We must therefore have $\alpha_2^-
\le \lam^+$, which completes the proof.
 \end{proof}

\begin{lemma}\label{lem:decrease_alpha2}
Assume $(x,X,y,Y_{12},\alpha)$ lacks only (\ref{equ:4x4first}) with
$\alpha_1=0$, and $\alpha_2 > \alpha_2^+$. Then
$(x,X,y,Y_{12},\alphahat)$ satisfies (\ref{equ:nobeta}) with $\alphahat
= (0,\alpha_2^+)$.
\end{lemma}

\begin{proof}
We follow a similar proof as for the preceding lemma. In this case,
however, since we are decreasing $\alpha_2$ to $\alpha_2^+$, we need
to show $\alpha_2^+ \ge x_1+x_2 -X_{12}-Y_{12}$ and $\alpha_2^+ \ge
\lam^-$, where $\lam^- = \tfrac12(1 - \rho)(y_2 - Y_{12})$ as defined
in Lemma \ref{lem:psdimpliespsd}. Note that $\alpha_2 \ge \lam^-$ because
$(x,X,y,Y_{12},\alpha)$ lacks only (\ref{equ:4x4first}), just as in the
preceding lemma.

For the first inequality, from \eqref{eq:alpha2+} it suffices to show
\[
x_2-  \frac{X_{12}x_1}{X_{11}} \ge  x_1+x_2 -X_{12}-Y_{12}
\]
which is equivalent to
\[
 X_{12}x_1+  X_{11}x_1 -X_{11}X_{12} \le
X_{11}Y_{12}.
\]
Since $\theta = Y_{12}X_{11} - x_1^2 \ge 0$, it thus suffices to show
\begin{eqnarray*}
    X_{12}x_1+  X_{11}x_1 -X_{11}X_{12} &\le& x_1^2 \\
    X_{12}(x_1-X_{11}) &\le& x_1(x_1-X_{11}),
\end{eqnarray*}
which certainly holds because $X_{12} \le x_1$ and $X_{11} \le x_1$.

For the second inequality, assume by contradiction that $\alpha_2^+ < \lam^-$. We claim $\alpha_2^*
< \lam^-$, which by \eqref{eq:alpha2*} is equivalent to
\[
    x_2 - \frac{X_{12} x_1}{X_{11}} < \tfrac12 (1 - \rho)\left(y_2 - \frac{x_1^2}{X_{11}} \right).
\]
From \eqref{eq:alpha2+}, the assumption $\alpha_2^+ < \lam^-$, and the inequality $Y_{12} X_{11} \ge x_1^2$,
we have
\[
    x_2 - \frac{X_{12} x_1}{X_{11}} \le \alpha_2^+ < \lam^-
    = \tfrac12(1 - \rho)(y_2 - Y_{12}) \le \tfrac12(1 - \rho)\left(y_2 - \frac{x_1^2}{X_{11}}\right),
\]
as desired. Since \eqref{equ:2gdet} holds at $\alpha_2 \ge \lam^-$ and
$\alpha_2^*<\lam^-$, the determinant in \eqref{equ:2gdet} is strictly positive at $\lam^-$, which implies that (\ref{equ:5x5})
holds with $\alpha_2$ replaced by $\lam^-$; the logic is identical to that for $\lam^+$ in the proof of Lemma \ref{lem:increase_alpha2}.
Then Lemma \ref{lem:psdimpliespsd} implies that (\ref{equ:4x4first}) holds with $\alpha_2$ replaced by
$\lam^-$, contradicting the
assumption that $\alpha_2^+<\lam^-$, so in fact $\alpha_2^+\ge \lam^-$.
\end{proof}

\subsection{Removing (\ref{equ:4x4first}) and (\ref{equ:4x4second}) does not
affect the projection} \label{sec:removing4x4}

We can now prove the following streamlined version of
Theorem \ref{thm:nobeta}, which requires only one of the four PSD
conditions \eqref{equ:3x3}--\eqref{equ:5x5}.

\begin{theorem} \label{thm:improved}
$\Hcal$ equals the projection onto $(x,X,y,Y_{12})$ of $(x, X, y, Y_{12},\alpha)$ satisfying the convex
constraints (\ref{equ:nobeta:lin1})--(\ref{equ:nobeta:lin4}) and (\ref{equ:5x5}).
\end{theorem}

\begin{proof} We must show that if $(x,X,y,Y_{12},\alpha)$
satisfies (\ref{equ:nobeta:lin1})--(\ref{equ:nobeta:lin4}) and (\ref{equ:5x5}),
then $(x,X,y,Y_{12})\in\Hcal$.  By Theorem \ref{thm:nobeta} this is equivalent to
showing that there is an $\alphaprime$ so that $(x,X,y,Y_{12},\alphaprime)$
satisfies all of the constraints in \eqref{equ:nobeta}.

If (\ref{equ:nobeta:lin1})--(\ref{equ:nobeta:lin4}) are satisfied, then
(\ref{equ:3x3}) is redundant by
Proposition \ref{pro:3x3redundant}. Moreover, as described above Lemma \ref{lem:lacks},
if \eqref{equ:5x5} also holds then at most one of \eqref{equ:4x4first}--\eqref{equ:4x4second} can fail to hold.
If both \eqref{equ:4x4first}--\eqref{equ:4x4second} hold then there is nothing to show, so we assume
without loss of generality that \eqref{equ:4x4first} fails to hold; that is,
$(x,X,y,Y_{12},\alpha)$ lacks only \eqref{equ:4x4first}.

Assume first that $\alpha_1=0$. If $\alpha_2<\alpha_2^-$, then by Lemma
\ref{lem:increase_alpha2} we know that $(x,X,y,Y_{12},\alphahat)$ satisfies
\eqref{equ:nobeta}, where $\alphahat=(0,\alpha_2^-)$.  Similarly, if $\alpha_2>\alpha_2^+$, then by
Lemma \ref{lem:decrease_alpha2} we have the same conclusion using $\alphahat=(0,\alpha_2^+)$.
Therefore $(x,X,y,Y_{12})\in\Hcal$.

If $\alpha_1>0$ we apply the transformation in Lemma \ref{lem:lacks}
to obtain $(\xbar,\Xbar,y,Y_{12},\alphabar)$, with $\alphabar=(0,\alpha_2)$,
that lacks only \eqref{equ:4x4first}.  We then apply either Lemma \ref{lem:increase_alpha2}
or Lemma \ref{lem:decrease_alpha2} to obtain $\alphahat=(0,\alphahat_2)$ so that $(\xbar,\Xbar,y,Y_{12},\alphahat)$
satisfies \eqref{equ:nobeta}.  Let $\alphaprime=(\alpha_1,\alphahat_2)$. We claim that $(x,X,y,Y_{12},\alphaprime)$
satisfies \eqref{equ:nobeta} as well. For the linear conditions \eqref{equ:nobeta:lin1}--\eqref{equ:nobeta:lin4} this is
immediate from the facts that both $(x,X,y,Y_{12},\alpha)$ and $(\xbar,\Xbar,y,Y_{12},\alphahat)$ satisfy
\eqref{equ:nobeta:lin1}--\eqref{equ:nobeta:lin4}, and $\xbar_1-\alphabar_1=x_1-\alpha_1$.  Therefore
\eqref{equ:3x3} is also satisfied at $(x,X,y,Y_{12},\alphaprime)$.  The fact that the remaining PSD conditions
\eqref{equ:4x4first}--\eqref{equ:5x5} are satisfied at $(x,X,y,Y_{12},\alphaprime)$ follows from the facts that these conditions are
satisfied at $(\xbar,\Xbar,y,Y_{12},\alphahat)$, $\xbar_1-\alphabar_1=x_1-\alpha_1$, the definition of $\Xbar_{11}$ and the
Schur complement condition.
\end{proof}

\section{Another interpretation}\label{sec:conjecture}

The representation for $\Hcal$ in Theorem \ref{thm:improved}
was obtained by starting with the representation in Theorem
\ref{thm:nobeta} and then arguing that only the single semidefiniteness
constraint \eqref{equ:5x5} was necessary. In this section we
describe an alternative derivation for the representation in Theorem
\ref{thm:improved}. This derivation provides another
interpretation for the conditions of Theorem \ref{thm:improved} and also
leads to a simple conjecture for a representation of $\Hcal^\prime$ as
defined in the Introduction.

The alternative derivation is based on replacing the variables $y$
with $t=e-y$, as was done for the case $n=1$ in the proof of Theorem
\ref{thm:n=1}. Note that each $y_i$ is binary if and only if $t_i$
is binary, and $(y, Y_{12}) \in\RLT_y$ if and only if $(t,T_{12}) \in\RLT_y$ where
$T_{12} = 1 + Y_{12} - y_1 - y_2$. In fact the
linear constraints \eqref{equ:nobeta:lin1}--\eqref{equ:nobeta:lin4}
can be obtained by considering the equations $x_i+s_i+t_i=1$,
$i=1,2$, generating RLT constraints by multiplying each equation in
turn by the variables $(x_j,s_j,t_j)$, $ i=1,2$, and then projecting
onto the variables $(x,X,t,T_{12},\alpha)$, where $\alpha_1\approx
x_1t_2= x_1(1-y_2)$, $\alpha_2\approx x_2t_1=x_2(1-y_1)$,
$T_{12}=1+Y_{12}-y_1-y_2 \approx t_1t_2$. Substituting variables and
applying a symmetric transformation that preserves semidefiniteness, the
PSD condition \eqref{equ:5x5'} can be written in the form
\begin{equation}\label{equ:strengthenedPSD}
\begin{pmatrix}
1-T_{12} & x_1 & x_2 & t_1-T_{12} & t_2-T_{12} \\
x_1 & X_{11} & X_{12} & 0 & \alpha_1 \\
x_2 & X_{12} & X_{22} & \alpha_2 & 0 \\
t_1-T_{12} & 0 & \alpha_2 & t_1-T_{12} & 0\\
t_2-T_{12} & \alpha_1 & 0 & 0 & t_2-T_{12}
\end{pmatrix}\gesem 0.
\end{equation}
The PSD constraint \eqref{equ:strengthenedPSD} has a simple interpretation as a strengthening of the natural PSD
condition
\begin{equation}\label{equ:originalPSD}
\begin{pmatrix}
1 & x_1 & x_2 & t_1& t_2\\
x_1 & X_{11} & X_{12} & 0 & \alpha_1 \\
x_2 & X_{12} & X_{22} & \alpha_2 & 0 \\
t_1 & 0 & \alpha_2 & t_1 &T_{12} \\
t_2 & \alpha_1 & 0 & T_{12} & t_2
\end{pmatrix}\gesem 0.
\end{equation}
The matrix in \eqref{equ:strengthenedPSD} is obtained from the matrix in \eqref{equ:originalPSD} by
subtracting $T_{12}uu\tran$, where $u=(1,0,0,1,1)^T$. This can be interpreted as removing the portion of the matrix
corresponding to $t=e$, or equivalently $y=0$, if the matrix in \eqref{equ:originalPSD} is decomposed into a
convex combination of four matrices corresponding to $t \in \{0, e_1, e_2, e\}$, similar to the
decomposition of $\Hcal$ into a convex combination of $\Hcal_y$, $y\in\set{0,e_1,e_2,e}$ in Section \ref{sec:disj}.
Note in particular that $T_{12}=\lam_0$, as defined in \eqref{equ:lam_y}.

We know that to obtain a representation of $\Hcal$ the condition \eqref{equ:strengthenedPSD} cannot be replaced by
\eqref{equ:originalPSD}; there are solutions $(x,X,y,Y_{12},\alpha)$ that are feasible with the weaker PSD condition
but where $(x,X,y,Y_{12})\notin\Hcal$.  However it appears that the condition \eqref{equ:originalPSD} is
sufficient to obtain a representation of $\Hcal^\prime$.
The following conjecture regarding $\Hcal^\prime$ is supported by extensive numerical computations, but remains
unproved.

\begin{conjecture}\label{con:Hcalprime}
 $\Hcal^\prime$ equals the projection onto $(x,X,y)$ of $(x,X,y,Y_{12},\alpha)$ satisfying the
constraints \eqref{equ:nobeta:lin1}--\eqref{equ:nobeta:lin4} and \eqref{equ:originalPSD}, where $t_1=1-y_1$,
$t_2=1-y_2$ and $T_{12}=1+Y_{12}-y_1-y_2$.
\end{conjecture}

\noindent Note that \eqref{equ:nobeta:lin1}--\eqref{equ:nobeta:lin4}
and (\ref{equ:originalPSD}) amount to the relaxation of $(x,xx^T,y)$,
which enforces PSD and RLT in the $(x,X,y,Y_{12})$ space and also
exploits the binary nature of $y$. In other words, the standard approach
for creating a strong SDP relaxation would be sufficient to capture the
convex hull of $(x,X,y)$ in this case, similar to the case of $n=1$
as shown in the proof of  Theorem \ref{thm:n=1}, as well as the
characterization of QPB for $n=2$ from \cite{Anstreicher.Burer.2010}. 

\bibliographystyle{abbrv}
\bibliography{switching}

\end{document}